\documentclass[smallextended]{amsart}       % onecolumn (second format)
\usepackage{graphicx,hyperref,colortbl,amssymb,adjustbox}
\newtheorem{theorem}{Theorem}
\newtheorem{conjecture}{Conjecture}

\begin{document}

\title{Some remarks on the Game of Cycles}%\thanks{}


\author{Robbert Fokkink \and  Jonathan Zandee} 
%\dedicatory{}

\begin{abstract}
The Game of Cycles is an impartial game on a planar graph
that was introduced by Francis Su. In this short note we
address some questions that have been raised on the game,
and raise some further questions. We assume that the reader
is familiar with basic notions from combinatorial game theory.
\end{abstract}
\maketitle

\section{The rules of the game}
\label{intro}

The Game of Cycles was invented by Francis Su~\cite{Su}. It is an impartial
game that is played on a planar graph $\Gamma\subset\mathbb R^2$, which is
connected and simple. 
The complement $\mathbb R^2\setminus\Gamma$ falls apart into a finite number of domains. The 
edges around the bounded domains are called the 
\emph{cells}.
Initially $\Gamma$ is undirected. Alice and Bob take turns and direct one hitherto 	undirected edge. 
It is not allowed to create a sink (a vertex
with all edges pointing inward) or a source (all edges pointing outward). 
If a player succeeds in forming a \emph{cycle cell} in which all
edges point in the same direction (either clockwise or anti-clockwise), then
that player wins. Otherwise the last player to make a move wins. Given a graph, the problem is to determine whether it is winning for Alice or winning for Bob. This problem has been solved for some specific graphs in~\cite{Cycles, Barua, Lin, Mathews}.

It is possible to remove loose ends from the graph.
Suppose a vertex~$v$ is incident with only one edge~$e$. 
Then~$e$ is \emph{unmarkable} because if it is directed, then $v$ is a source or a sink.
Since $e$ is never directed, the sink/source restriction never applies to
its other vertex $w$. We might as
well remove $v$ and $e$ from the graph, and declare $w$ to be a special vertex,
which is allowed to be a source or a sink
(this operation of removing unmarkable edges
is called \emph{trimming} in~\cite{Mathews}).
\begin{figure}[htbp]
	\centering
		\includegraphics[width=0.60\textwidth]{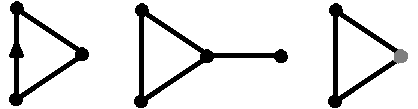}
	\caption{On the left, Alice has won the game by directing an edge of the triangle.
	Every move for Bob loses. The lollipop graph in the center is equivalent to the triangle
	with one special vertex on the right. 
        Alice wins by the same move, 
        but if she directs
	one of the other two edges, she loses. The lollipop graph has Grundy value 2.}
	\label{fig:lollipop}
\end{figure}
We only consider professional players. If a player wins by forming a cycle
cell, then that can only happen if the other player had no choice on the previous move. 
A professional player would have resigned. 
Therefore, we do not really change the game if we do not allow
moves that enable
the closure of a cycle cell (called \emph{death moves} in~\cite{Cycles}). This has the pleasing effect that the game is
over once there are no more moves, i.e., it satisfies the normal
play condition, as in a standard impartial game~\cite{Albert}.

In our (equivalent) version of the Game of Cycles, we allow special vertices and do not allow moves that
enable cycle cells. It goes without saying that a vertex is special if it has degree~$1$. 
The results below are 
based on the work
of the second author~\cite{	Zandee}.
 
\section{Previous results}

The Game of Cycles was analysed by Ryan Alvarado et al~\cite{Cycles}, who were
able to find winning strategies for graphs $\Gamma$ that have certain symmetries.
These winning strategies are copycat strategies. Either Bob copies the moves
of Alice under the symmetry, or Alice makes a special first move on a unique edge
that is fixed by the symmetry, and then she copies Bob's moves. 
More specifically, suppose $h$ is a graph isomorphism of $\Gamma$ such that $h^2$ is
the identity (an \emph{involution}) that fixes at most one edge.
If none of the edges is fixed, then Bob is the copycat. If only one edge is fixed,
then Alice marks it and she is the copycat.
The idea behind the copycat strategy is that if $e$ can be marked
then $h(e)$ can be marked,
and so the copycat always has the last move.
Copycats mark
 \emph{in the opposite direction}.
If the opponent directs~$e$ 
from vertex $v$ to vertex $w$, then the copycat directs $h(e)$ from $h(w)$ to $h(v)$. 
In this way,
the copycat avoids sinks and sources. 
The copycat also needs to avoid death moves.
If each cell is either invariant or disjoint from itself under $h$, then
Alvarado et al~\cite{Cycles} show that the copycat's move is never a death move. 
Thus the copycat wins the game if such an involution $h$ exists. 
\begin{figure}[htbp]
	\centering
		\includegraphics[width=0.7\textwidth]{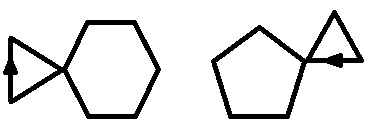}
	\caption{Winning moves for Alice for the triangle connected to an $n$-gon. If $n$
	is even, Alice marks the unconnected edge of the triangle and wins by a copycat
        strategy (this is the unique fixed edge under a reflection). 
	If $n$ is odd, Alice marks a connected edge. Now the unconnected edge of the triangle is unmarkable.
        Alice prevents the connecting vertex between $n$-gon and triangle from becoming a source or sink,
        and wins the game.}
	\label{fig:ngon}
\end{figure}

For all graphs that are solved in \cite{Cycles},
the \emph{size} of the graph (number of edges)
determines who wins the game. Alice wins if the size is odd
and Bob wins if the size is even. One of the questions in~\cite{Cycles} was
whether this is true for all graphs. 
Some care is required, as the lollipop graph has size $4$
but is won for Alice.
Kailee Lin~\cite{Lin} specified the question and asked: 
is a graph is winning for Alice if and only
if its number of markable edges is odd?
She established that this \emph{parity conjecture} 
is true for general lollipop graphs
($n$-gons with loose ends attached). 
All games that are solved in~\cite{Cycles, Barua, Lin} satisfy
this conjecture.
However, Leah Karker and Shanise Walker~\cite{Su2} found a counterexample:
a triangle connected to an $n$-gon is winning for Alice for all $n\geq 4$,
see Fig~\ref{fig:ngon}.
We return to this example at the end of our paper.

A tree has no cycle cells, so only the source/sink restriction applies.
For tree, the game may seem simple, but it turns out to be challenging and even
spider graphs (only one vertex of degree $> 2$) are non-trivial.
Bryant Mathews~\cite{Mathews} showed that if $\Gamma$ is a three-legged spider,
then its Grundy value is zero if all legs of $\Gamma$ are even (note that we remove loose
ends, so in our game all edges in a tree are markable). The proof is long 
and falls apart into many different cases. For trees the parity conjecture may hold.

\begin{conjecture}[parity conjecture for trees]\label{conj1}
The Grundy value of a tree is zero if and only if its size is even.
\end{conjecture}

\begin{figure}[ht]
	\centering
		\includegraphics[width=0.50\textwidth]{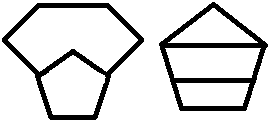}
	\caption{Some sweet graphs. The ice cream cone on the left has size ten and Grundy value zero. The layered cake on the
	right has size nine and Grundy value one.}
	\label{fig:suchallenge}
\end{figure}
Alvarado et al~\cite{Cycles} left two computational challenges for the reader, as illustrated
in Fig~\ref{fig:suchallenge}. They both satisfy the parity conjecture.

\section{New examples}

The butterfly graph 
can be embedded in two ways in the plane: open wings (standard) or closed wings (one triangle
bounded by the other), see Fig~\ref{fig:butterfly}. 
If the butterfly closes its wings, the inner cell is a part of the outer cell and
therefore only the inner cell matters. 
\begin{figure}[h]
	\centering
		\includegraphics[width=0.60\textwidth]{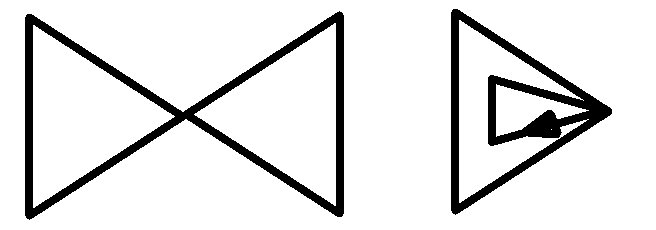}
		\caption{The butterfly graph (left) is winning for Bob, by a copycat strategy.
		If the butterfly closes its wings (right), then Alice wins 
        by directing the marked edge on the inner wing, which
		makes the adjacent edge
		on that wing unplayable. Four playable edges remain and Alice wins as the second player
		in a game on a graph of size four.}
	\label{fig:butterfly}
\end{figure}
The Grundy number of the
open winged butterfly is zero. 
The Grundy number of the closed winged butterfly is three. A graph can be winning or losing 
(for Alice) depending on how it is embedded in the plane.

The wedge of $n$-gons in Fig~\ref{fig:ngon} and the butterfly graph in Fig~\ref{fig:butterfly}
can be disconnected by removing a single vertex.
They are not $2$-connected. 
One of the questions that we asked ourselves is: does the parity conjecture hold for
$2$-connected graphs?
It turns out that it does not.
There exists a $2$-connected
graph of odd size that is winning for Bob, as illustrated in~Fig~\ref{fig:size9}. 
It is invariant under reflection in
the central edge. Note that the central two cells are not disjoint
under this involution, so the copycat result of~\cite{Cycles} does not
apply.
Indeed, the first move in the copycat strategy directs the central edge.
However, Bob counters by directing the left outermost edge in the opposite direction.
After this move, the remaining three edges in the left hand box are unplayable.
Only the edges in the right hand box remain and Bob wins as the
second player on a graph of size four. 
Alice essentially has three other first moves, but
they are all losing. 
The full analysis of this graph is too elaborate to reproduce here
and can be found in~\cite{Zandee}.
\begin{figure}[hb]
	\centering
		\includegraphics[width=0.70\textwidth]{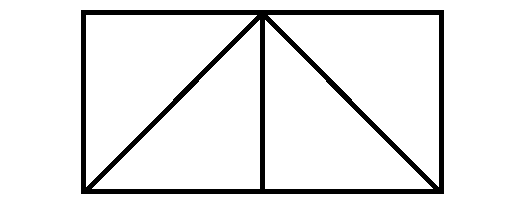}
		\caption{This graph of size nine has Grundy number zero and is
			a counterexample to the parity conjecture.}
	\label{fig:size9}
\end{figure}

We say that a tree is \emph{branching} if all internal vertices have degee~$>2$.

\begin{theorem}
A branching tree is winning for Alice if and only if its size is odd.
Its Grundy number is~$1$ in that case.
\end{theorem}

\begin{proof}
Essentially, we have a take-away game in which players remove chips one by one.
The winning player needs to avoid sinks and sources at internal vertices.
On her first move, Alice directs an edge at an inner vertex. We declare
that vertex (either one, if both are inner vertices)
to be the \emph{root} of the tree. Thus we have ordered the tree
and can speak about parent and child vertices.
We say that a parent with an odd number of children is an \emph{odd parent}.
The parity of $p$, the number of odd parents, equals
the parity of the size of the tree.
If an edge is directed then we think of this as breaking the tie
between parent and child, as if the edge is removed. 
Because of this the parity of $p$ changes with each move in the game.

The winning player applies a copycat strategy. If the size of the tree is even,
then Bob is the winning player. If Alice directs the edge of an even parent,
then Bob counters and directs another edge of that parent in the
opposite direction. The sink/source condition now is no longer relevant
for this vertex. If Alice directs an 
edge of an odd parent, then Bob also selects an edge
of an odd parent (he can, $p$ is odd on his moves). 
Each internal vertex $v$ will go through the stage of an even number of children.
In that case, an edge between $v$ and its children will first be directed
by Alice, and immediately countered by Bob. Thereupon the sink/source restriction will be lifted. 
Therefore, Bob always has a move even if one child remains.
He wins the game.
If the size of the tree is odd, then Alice is the winning player by the
exact same argument, because now $p$ is odd whenever she has the move.

Note that we allowed an arbitrary first move of Alice, and used it to
select a root of the tree. If the size of the tree is odd, Alice wins
the game regardless of her first move. Therefore, the Grundy value
is one if the size is odd.
If the size is even, then Bob wins and the Grundy value is zero.
\end{proof}

The parity conjecture holds for branching trees. The difficulty lies in trees
with many non-branching vertices, such as spiders.

\section{Grundy numbers}

There does not seem to be a straightforward
graph invariant to decide if a graph is winning for Alice or for Bob.
The only possible algorithmic approach to the game seems to be by standard backward induction,
which is only feasible for small graphs (size up to around ten).
It is not hard to find positions in the game
(that is, graphs with some directed edges) of arbitrary Grundy numbers,
see~\cite{Barua, Lin, Mathews, Zandee}, but
it is a computational challenge to find a graph with a large Grundy value. 
We did not get very far.
The maximum
Grundy number that we were able to find is 3, see Fig~\ref{fig:mill}, \ref{fig:wedged}, and~\ref{fig:boxes} below. 
Can anybody find a graph
of Grundy number four or more?

\begin{figure}[htbp]
	\centering
		\includegraphics[width=1.00\textwidth]{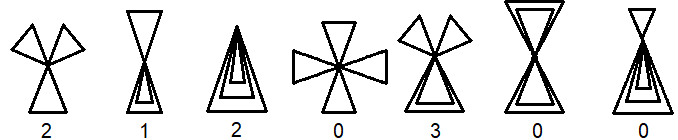}
		\caption{Windmill graphs and their Grundy numbers. }
	\label{fig:mill}
\end{figure}

\begin{figure}[htbp]
	\centering
		\includegraphics[width=0.90\textwidth]{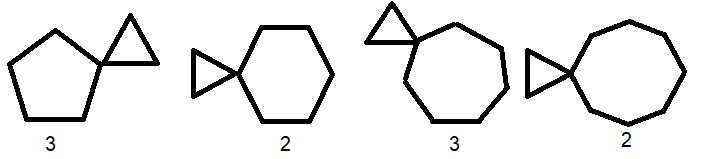}
  		\caption{Fishy graphs: wedges of $n$-gons and a triangle. }
	\label{fig:wedged}
\end{figure}

\begin{figure}[htbp]
	\centering
		\includegraphics[width=0.90\textwidth]{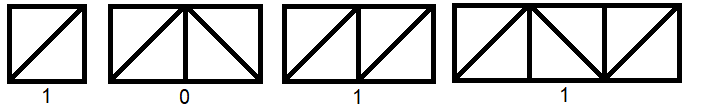}
  		\caption{Box graphs and their Grundy numbers. }
	\label{fig:boxes}
\end{figure}

Leah Karker and Shanise Walker showed that the
fishy graphs in Fig~\ref{fig:wedged} are all winning for Alice. 
Our computational results suggest that their Grundy numbers are $2$ if $n$ is even and
$3$ if $n$ is odd. We prove that these are the Grundy numbers if the connecting
vertex is special.

\begin{theorem}
The Grundy number of an $n$-gon with a special vertex is zero if $n$ is even and
one if $n$ is odd, expect for the triangle which has Grundy number two.
\end{theorem}

\begin{proof}
Place the $n$-gon in the plane such that it is invariant under reflection 
in the $x$-axis and such that the special vertex is on this axis. This reflection
is an involution and the copycat strategy works. 
If the opponent marks one edge of the special vertex, then the copycat responds
by marking its other edge, so the special vertex is not really special for the copycat.
Bob wins if $n$ is even and
Alice wins if $n$ is odd. We only have to determine by `how much' Alice wins.

If $n=3$ then we have the graph of Fig~\ref{fig:lollipop}.
Alice has a winning move and a losing move, which implies that the Grundy number is two. The case
of odd $n>3$ remains. We prove that it is losing for Alice if we add
the option of a \emph{pass}. This option is available once and only once. If one of the
players passes, then the pass is off the table. 
The reflection in 
the $x$-axis fixes one edge denoted $f$.
We modify the reflection by defining $h(f)$ to be the pass (and vice versa).
If Alice marks $e$ then Bob marks $h(e)$
in the opposite direction. If Alice passes, then Bob marks $f$ in any possible direction.
Observe that $h(e)$ is unmarked and that an adjacent edge of $h(e)$ is directed
if and only if its corresponding adjacent edge at $e$ is directed.
Also note that the adjacent edges of $f$ are copies under $h$.
If Bob needs to direct $f$, either both of its adjacent edges are unmarked or
they point in the same direction. 
Therefore Bob can direct $f$ if Alice passes.
However, there is a catch. Directing $h(e)$ as prescribed may be a death move.
If Alice keeps directing edges other than $f$
clockwise, then Bob follows suit until only $f$ remains. Alice then wins by
closing the cycle. In this case, Bob's final move is a death move, which is
unprofessional. 
We need to modify copycat into a professional strategy: if directing $h(e)$
as prescribed is a death move, then Bob takes a pass.
Is this possible? If $h(e)$ is a death move, then $n-2$ edges are marked
when it is Bob's move. If the pass is unavailable, then the game has gone through
$n-1$ moves, which is an even number. However, if Bob moves, then the 
game has gone through an odd number of moves,
so this cannot happen. If $n-2$ edges are marked and Bob has the move, then he may pass.
\begin{figure}[hb]
	\centering
		\includegraphics[width=0.60\textwidth]{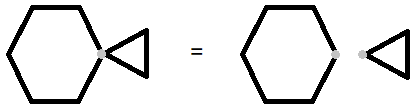}
  		\caption{The game on a wedge of two $n$-gons at a special vertex is a sum of two games. }
	\label{fig:sum}
\end{figure}
Now two edges remain for Alice,
one of which is $f$. It has no special vertex and at least one of its adjacent
edges has to be marked. It can only be directed in one way and that is a 
death move. So Alice cannot mark $f$.
The other remaining edge is $h(e)$.
If it has no special vertex, it can only be directed in one way for the same reason as $f$.
This is a death move and not allowed.
If $h(e)$ has a special vertex, then it is adjacent to $e$ at this vertex.
The other adjacent edge has already been marked (here we need $n>3$ to rule out
that $f$ is adjacent to $h(e)$).
Again, $h(e)$ only admits a death move.
After Bob passes, Alice is out of moves.
It follows that this game with one pass is winning for Bob. Since one pass is equivalent
to a game of Grundy number 1, we conclude that odd $n$-gons with $n>3$
have Grundy number~$1$.
\end{proof}

\section{Acknowledgement}

We would like to thank Francis Su for a helpful communication.


\noindent
Institute of Applied Mathematics\\
Delft University of Technology\\
Mekelweg 4 \\
2628 CD Delft, The Netherlands\\
\texttt{r.j.fokkink@tudelft.nl; j.s.zandee@student.tudelft.nl}

\end{document}